\documentclass[12pt, oneside]{amsart}

\usepackage{amsmath, amsfonts, amssymb, amsthm}

\usepackage[utf8]{inputenc}

\usepackage{graphicx}

\renewcommand{\part}[2]{\frac{\partial #1}{\partial #2}}

\newcommand{\R}{\mathbb{R}}

\newcommand{\A}{\textrm{Area}}
\newcommand{\bs}{\backslash}

\theoremstyle{definition}

\newtheorem{thm}{Theorem}[section]
\newtheorem{definition}[thm]{Definition}
\newtheorem{lem}[thm]{Lemma}

\newtheorem{cor}[thm]{Corollary}

\newtheorem{rmk}[thm]{Remark}
\newtheorem{conj}[thm]{Conjecture}

\title[Minimal Surface Tangent Cones at Infinity]{A Criterion for Uniqueness of Tangent Cones at Infinity for Minimal Surfaces}
\author{Paul Gallagher}
\date{March 20, 2017}

\begin{document}

\begin{abstract}
We partially resolve a conjecture of Meeks on the asymptotic behavior of minimal surfaces in $\mathbb{R}^3$ with quadratic area growth.
\end{abstract}
\maketitle

\section{Introduction}

Let $\Sigma$ be an embedded minimal surface in $\R^3$. By the monotonicity formula, the area density
\begin{equation*}
    \Theta(r) := \frac{A(\Sigma\cap B_r)}{\pi r^2}
\end{equation*}
is nondecreasing. If 
    \begin{equation*}
        \lim_{r\to\infty} \Theta(r) = \Theta(\infty) = k < \infty,
    \end{equation*}
we say that $\Sigma$ has quadratic area growth, or the area growth of $k$ planes.

\begin{figure}
    \centering
    \includegraphics[scale=0.75]{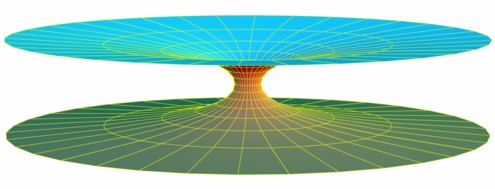}
    \caption{Catenoid (from http://www.indiana.edu/~minimal)}
    \label{fig:catenoid}
\end{figure}

\begin{figure}
    \centering
    \includegraphics{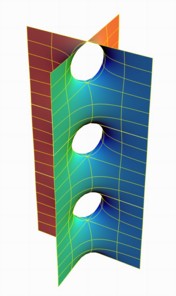}
    \caption{Scherk Singly Periodic (from http://www.indiana.edu/~minimal)}    
    \label{fig:scherk}
\end{figure}

\begin{figure}
    \centering
    \includegraphics[scale=0.5]{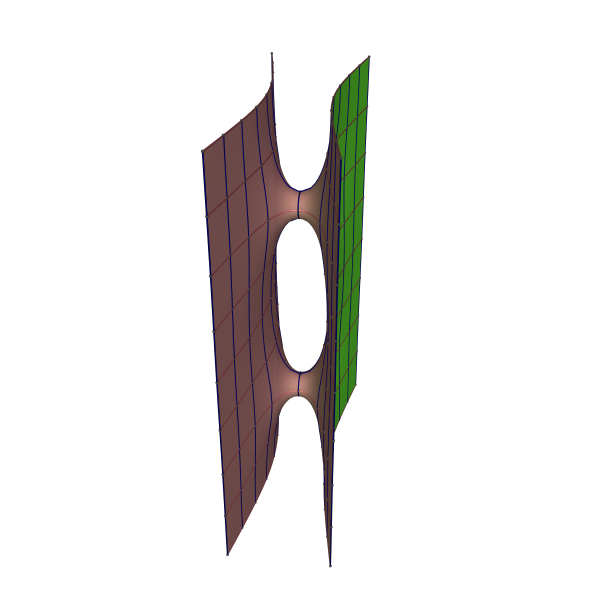}
    \caption{Non-orthogonal Scherk (from http://www.indiana.edu/~minimal)}
    \label{fig:scherk1}
\end{figure}

For surfaces with the growth of 2 planes, there are two canonical examples: the catenoid (Fig \ref{fig:catenoid}), and Scherk's singly periodic surfaces, which occur in a one parameter family (Fig \ref{fig:scherk} and Fig \ref{fig:scherk1}), where the parameter is the angle betwee the two leaves. As the angle goes to zero, the Scherk surfaces approach a catenoid on compact sets after an appropriate rescaling. In 2005, Meeks and Wolf proved the following theorem:

\begin{thm}\cite{meeks+wolf}\label{thm:periodic}
Suppose that $\Sigma$ is an embedded minimal surface in $\R^3$ which has infinite symmetry group and $\Theta(\infty) < 3$. Then $\Sigma$ is either a catenoid or a Scherk example. 
\end{thm}

Meeks has conjectured that the symmetry condition in the above may be removed:

\begin{conj}\cite{meeksGlobal}
Let $\Sigma$ be an embedded minimal surface in $\R^3$ with area growth of 2 planes. Then $\Sigma$ is either a catenoid or a Scherk example.
\end{conj}

However, an initial difficulty with the above is that it is not yet known that a minimal surface with quadratic growth even needs to be asymptotic to a catenoid or a Scherk example. By the compactness theory from GMT, it is known that if $\Sigma$ is an embedded minimal surface with quadratic area growth, then for any sequence $r_i\to\infty$, there exists a subsequence $\rho_i$ such that $\Sigma/\rho_i\cap B_1$ converges to a minimal cone $\mathcal{C}$ in the varifold topology. Such a cone $\mathcal{C}$ is called a \textbf{tangent cone at infinity}. A priori, there may be many tangent cones at infinity. 

This leads to the following conjecture, also due to Meeks:

\begin{conj}\cite{meeksGlobal}
Let $\Sigma$ be an embedded minimal surface in $\R^3$ with quadratic area growth. Then $\Sigma$ has a unique tangent cone at infinity.
\end{conj}

In the case of finite genus, this had already been resolved by Collin \cite{collin}, who proved that any minimal surface with finite genus and quadratic area growth must be asymptotic to a single multiplicity $k$ plane. In particular, when combined with a result of Schoen \cite{schoen}, this resolves Meeks' full conjecture in the case of finite genus - that is, the only minimal surface with the area growth of two planes and finite genus is the catenoid. 

In this paper, we prove that Meeks' conjecture holds true under additional assumptions:

\begin{thm}\label{thm:sublin1}
Let $\Sigma$ be an embedded minimal surface with the area growth of $k$ planes. Suppose that there exists $\alpha < 1$ such that for all $R$ sufficiently large, there exists a line $l_R$
    \begin{equation*}
        \Sigma\cap B_R \cap \{d(x, l_R) > R^{\alpha}\}
    \end{equation*}
is a union of at least $2k$ disks $\Sigma_i$ and such that $\partial\Sigma_i$ is homotopically nontrivial in $\partial(B_R \cap \{d(x, l_R) > R^{\alpha}\})$. Then $\Sigma$ has a unique tangent cone at infinity. 
\end{thm}

This leads to the following:

\begin{cor}\label{cor:sublin}
    Let $\Sigma$ be an embedded minimal surface with quadratic area growth. Let 
        \begin{equation*}
            \mathcal{C}_\alpha = \{x_1^2 + x_2^2 \leq R^{2\alpha}\}.
        \end{equation*}
    Then if for some $R_0$, $\Sigma \backslash (B_{R_0}\cup\mathcal{C}_\alpha)$ is a union of $2k$ topological disks $\Sigma_i$ each with finitely many boundary components, then $\Sigma$ has a unique tangent cone at infinity.
\end{cor}

Note that the corollary substitutes the homotopy requirement from the theorem for the existence of a single line around which we can base our sublinearly growing set.

\subsection{Acknowledgments}
The author would like to thank his advisor, William Minicozzi, as well as Jonathan Zhu, Frank Morgan, Ao Sun, and Nick Strehlke for their comments and suggestions throughout the writing of this paper. 

\section{Proof of Theorem \ref{thm:sublin1}}\label{sec:proof}

The proof of this begins with the following: 

\begin{lem}[Lower Area Bound]\label{lem:LAB}
	Suppose that $\Sigma$ satisfies the conditions of Theorem \ref{thm:sublin1}. Then for some $C=C(\Sigma)$
	\begin{equation*}
		\A(B_R\cap\Sigma) >k\pi R^2 - CR^{\alpha + 1}.
	\end{equation*}
\end{lem}
\begin{proof}
	We will work on each leaf $\Sigma_i$ separately, and the lemma will come from adding the area of all the leaves together. 
    
    First note that $B_R\cap \{d(x,l_R)>R^{\alpha}\} = T_R$ is a rotationally symmetric solid torus and (since $\Sigma_i$ is a disk), $\partial\Sigma_i$ is contractible in $T_R$. However, since $T_R$ is rotationally symmetric, the smallest spanning disk for any such curve has area at least that of a vertical cross section $C$. Any such vertical cross section consists of a half-circle of radius $R$ minus a strip of length $2R$ and width $CR^{\alpha}$. Thus, we have 
    
    \begin{equation*}
        A(\Sigma_i) \geq A(C) \geq \frac{\pi}{2}R^2 - CR^{\alpha + 1}
    \end{equation*}
\end{proof} 

\begin{rmk}
    Note that Lemma \ref{lem:LAB} implies that there are in fact exactly $2k$ disks in the statement of Theorem \ref{thm:sublin1}.
\end{rmk} 

We make a definition:

\begin{definition}
The \textbf{error} at scale $r$ of a minimal surfaces with area growth of $k$ planes is defined as
	\begin{equation*}
		e(r) = k - \frac{\A(\Sigma\cap B_r)}{\pi r^2}
	\end{equation*}
\end{definition}

Thus, the Lemma \ref{lem:LAB} is equivalent to the statement:

\begin{equation}
	e(r) \leq Cr^{\alpha - 1}
\end{equation}

We now apply an argument of Brian White \cite{white} to prove uniqueness of the tangent cone. 

\begin{lem}
	Let $\Sigma$ satisfy the following: $\exists R_0, \alpha < 1$ such that for $R_0 < r< \infty$, 
		\begin{equation}\label{eqn:errbd}
			e(r) < Cr^{1-\alpha}
		\end{equation}
	Then $\Sigma$ has a unique tangent cone at infinity. 
\end{lem}
\begin{proof}
	Define $F(z) = z/|z|$. Then note that $A(F(\Sigma\cap (B_r\bs B_s)))$ is equal to the area of the projection of $\Sigma\cap (B_r\bs B_s))$ onto the unit sphere. We will bound this area. We have:
		\begin{align*}
			A(F(\Sigma\cap (B_r\bs B_s))) & = \int_{\Sigma\cap B_r\bs B_s}\frac{|x^N|}{|x|^3} d\Sigma \\
						& \leq \left[\int_{\Sigma\cap B_r\bs B_s}\frac{|x^N|^2}{|x|^4}d\Sigma\right]^{1/2} 
												\left[\int_{\Sigma\cap B_r\bs B_s}\frac{1}{|x|^2} d\Sigma\right]^{1/2} \\
		\end{align*}
	By monotonicity, the term inside the first bracket is smaller than $e(s)$. Also, the term in the second bracket can be bounded by distance and area. Thus, we get that the above is smaller than 
		\begin{equation*}
			e(s)^{1/2}(s^{-2}A(B_r\cap\Sigma))^{1/2}
		\end{equation*}
	Now, by equation (\ref{eqn:errbd}), along with the fact that $A(B_r\cap\Sigma) < k\pi r^2$, we have that this is bounded by 
		\begin{equation*}
			Cs^{(\alpha-1)/2} \left[\left(\frac{r}{s}\right)^2 (r^{-2}A(B_r\cap\Sigma))\right]^{1/2} \leq C\frac{r}{s^{(1-\alpha)/2 + 1}}
		\end{equation*}
	Pick $s$ and $r$ such that $s\leq r \leq 2s$. Then 
		\begin{equation*}
			A(F(\Sigma\cap(B_r\bs B_s))) \leq Cs^{(\alpha-1)/2}
		\end{equation*}
	We then sum the above bound to see 
		\begin{align*}
			A(F(\Sigma\cap(B_{2^n r}\bs B_r))) & = \sum_{k = 1}^n A(F(\Sigma\cap(B_{2^k r}\bs B_{2^{k-1}r}))) \\
							& \leq C\sum_{k = 1}^n (2^k r)^{(\alpha-1)/2} \\
							& \leq \frac{C}{r^{(1-\alpha)/2}}\frac{1}{1-2^{(1-\alpha)/2}}
		\end{align*}
	As $r\to\infty$, this term goes to zero. Thus, the area of the projection of $\Sigma\bs B_r$ approaches zero as $r$ gets large, which means that the tangent cone must be unique. 
	
\end{proof}

\section{Proof of Corollary \ref{cor:sublin}}

Let $\Sigma_i$ be one of the components of $\Sigma\backslash(\mathcal{C}_{\alpha}\cup B_{R_0})$. Then note that the closure of $\Sigma_i$ in $\mathbb{R}^3$ must be conformally equivalent to $\overline{\mathbb{D}}^2$ with finitely many boundary points removed. Take a neighborhood $N$ of one of these missing boundary points which does not come close to any other missing boundary points. Then $N\subset \Sigma_i$ has exactly one boundary component. There are two options for the shape of $\partial N$. 

\begin{enumerate}
    \item The function $x_3|_{\partial N}$ is unbounded in both directions. 
    \item $x_3|_{\partial N}$ is bounded in one direction. 
\end{enumerate}

Note that $x_3$ cannot be bounded in both directions, as then $\partial N$ would be compact, which it is not. 

We temporarily assume that Option 1 occurs. Let $\gamma$ be the portion of $\partial N$ which is not on the boundary of $\mathcal{C}_{\alpha}\cup B_{R_0}$. Take $R_0$ larger so that $\gamma\subset B_{R_0}$. Let $R>>R_0$. Then some component of $N\cap B_R \cap C_{\alpha}^c$ will satisfy the conditions of Theorem \ref{thm:sublin1}. This implies that it is possible to prove the Lower Area Bound lemma for this component, and in particular, the area must be asymptotic to $\pi R^2/2$.

The following lemma will complete our proof:

\begin{lem}
    Under our assumptions, Option 2 is not possible. 
\end{lem}

\begin{proof}
    Suppose that Option 2 occurs. WLOG, let $x_3|_{\partial N}$ be bounded below by 0, and let $(x_1, x_2, 0)\in {\partial N}$ be the point at which that minimum is achieved. Let $\rho = (x_1^2 + x_2^2)^{1/2}$. Let $C$ be a catenoid where the radius of the center geodesic is strictly larger than $2\rho$. Then by a simple application of the maximum principle, $N$ must intersect $C$. In particular, this implies that $\inf_{\partial B_R} x_3|_N < C_0 + \log R$. 
    
    Now, consider a sequence of $R_i$ such that $\Sigma\cap B_{R_i}$ converges to a tangent cone at infinity. By compactness, $R_i^{-1} N\cap \partial B_{R_i}$ must either converge to a union of geodesics on $B_1$ or must disappear at infinity. However, due to the discussion of the previous paragraph, $N$ cannot disappear at infinty, and so must converge to a nontrivial union of geodesics $\Gamma_j$, possibly with endpoints at the north or south poles. 
    
    Let $p$ be a nonsmooth point on $\cup \Gamma_j$. Then there must exist a neighborhood $S$ of $p$ such that $|A|$ restricted to $S\cap R_i^{-1} N$ is unbounded as $i\to \infty$. However, since $N$ is a minimal disk with quadratic area growth bounds, $|A|(x)$ must be bounded by $C/d(x)$, where $d(x)$ is the distance of $x$ from the boundary of $N$. 
    
    Suppose that $p$ is not equal to the south pole. Then we can choose our neighborhood $S$ of $p$ to stay away from the $x_3$ axis, so we will have that $|A|<C$ uniformly on $S\cap R_i^{-1} N$. Suppose that $p$ is equal to  the south pole. Then by the assumption of Option 2, $\partial N$ is only contained in the region $x_3 \geq 0$. So, we can choose $S = B_{1/2}(p)$, and this implies the same uniform $|A|$ bound.
    
    Therefore, there will be no nonsmooth points of $\cup\Gamma_j$, which implies that $\Gamma_j$ consists of a single great circle passing through the north pole.
    
    In particular, this implies that there are some $\epsilon(R_i)\to 0$ such that the area of $R_i^{-1} N\cap B_1$ is greater than $\pi - \epsilon(R_i)$, where $\epsilon\to 0$ as $R_i\to\infty$. Thus, we have at least $2k$ components of $\Sigma\backslash \mathcal{C}_{\alpha}$, each of which has area growth at least $\pi R^2/2$ by the discussion of Option 1. However, since the global area growth is $k \pi R^2$, no component can have growth $\pi R^2$. 
\end{proof}

\section{Future Directions}\label{sec:fut}
There are several potential extensions of the work above. Theorem \ref{thm:sublin1} and Corollary \ref{cor:sublin} effectively assume that all tangent cones of $\Sigma$ are unions of planes \textit{with a common axis}. It is likely not significantly more difficult to show that the same result holds in the case when the one-dimensional singular set is more complicated, as long as away from a sublinearly growing neighborhood, $\Sigma$ is a union of disks. That is, we have the following as another potential step towards the resolution of Meeks' Conjecture:

\begin{conj}
    Let $\Sigma$ have the area growth of $k$ planes, and suppose that there exists a uniform $\alpha <1$ such that for each $R>R_0>>1$, the following is true: There exist line segments $L_i(R)$, $1\leq i\leq m(R) <M$ such that outside of an $\alpha-$sublinearly growing neighborhood of $\cup L_i(R)$, $\Sigma\cap B_R$ is a union of disks. Then $\Sigma$ has a unique tangent cone at infinity. 
\end{conj}


\begin{thebibliography}{1}

  \bibitem[C]{collin} P. Collin, \emph{Topologie et courbure des surfaces minimales proprement plonges de $\mathbb{R}^3$}, Ann. of Math., \textbf{(145)}2 (1997), 1-31.
	
	\bibitem[M]{meeksGlobal} William H. Meeks, III, \emph{Global Problems in Classical Minimal Surface Theory}, Global theory of minimal surfaces, Clay Math. Proc., vol. 2, Amer. Math. Soc., Providence, RI, 2005, p. 453-469.
	
	\bibitem[MW]{meeks+wolf} William H. Meeks, III and Michael Wolf, \emph{Minimal surfaces with the area growth of two planes: The case of infinite symmetry}, Journal of the AMS, \textbf{(20)}2 (2006), 441-465.
	
	\bibitem[S]{schoen} Richard Schoen, \emph{Uniqueness, Symmetry, and Embeddedness of Minimal Surfaces}, Journal of Differential Geometry, \textbf{(18)} (1983), 791-809.
	
	\bibitem[W]{white} Brian White, \emph{Tangent Cones to Two-Dimensional Area-Minimizing Integral Currents are Unique}, Duke Mathematics Journal, \textbf{(50)}1 (1983), 143-160.
	
\end{thebibliography}
\end{document}